\documentclass[amsart,11pt]{article}
\usepackage{color,graphicx,latexsym,amssymb,amsmath,amsfonts,amsthm,float}
\usepackage{url}
\usepackage{hyperref}
\hypersetup{colorlinks}
\usepackage{authblk}
\usepackage[numeric,initials,nobysame]{amsrefs}
\newtheorem{theorem}{Theorem}
\newtheorem{lemma}[theorem]{Lemma}

\newtheorem{corollary}{Corollary}

\theoremstyle{definition}
\newtheorem*{remark}{Remark}
\newtheorem{definition}{Definition}

\newcommand{\e}{\epsilon}

\newcommand{\E}{\mathbb E}
\newcommand{\PP}{\mathbb P}

\definecolor{db}{rgb}{0.1,0,0.75}
\definecolor{lm}{cmyk}{0 ,1,0,0}

\newcommand{\G}{G_{n,\sigma}}

\newcommand{\pr}{\mathbb P}

\newcommand{\old}[1]{}

\title{Permuted Random Walk Exits Typically in Linear Time}

\author[1]{Shirshendu Ganguly \thanks{sganguly@math.washington.edu}}
\author[2]{Yuval Peres \thanks{peres@microsoft.com}}
\affil[1]{University of Washington}
\affil[2]{Microsoft Research}

\begin{document}
\maketitle

\begin{abstract} 
Given a permutation $\sigma$ of the integers $\{-n,-n+1,\ldots,n\}$ we consider the Markov chain $X_{\sigma}$, which jumps from $k$ to $\sigma (k\pm1)$ equally likely if $k\neq -n,n$. We prove that the expected hitting time of $\{-n,n\}$ starting from any point is $\Theta(n)$ with high probability when $\sigma$ is a uniformly chosen permutation. We prove this by showing that with high probability, the digraph of allowed transitions is an Eulerian expander; we then utilize general estimates of hitting times in directed Eulerian expanders. 
\end{abstract}
%

\section{Introduction}We denote $\{-n,-n+1,\ldots,n\}$ by $[-n,n]$ and the space of permutations of $[-n,n]$ by 
$S_{[-n,n]}$. We define the permuted walk on the interval $[-n,n]$ as follows:

\begin{definition} Given $\sigma\in S_{[-n,n]}$, the permuted walk $X_{\sigma}$ is defined as the Markov chain on $[-n,n]$ which goes from $x$ to $\sigma(x-1)$ and $\sigma(x+1)$ with probability $1/2$ each, if $x\notin \{-n,n\}$. The Markov chain jumps from $-n$ to $\sigma(-n)$ or $\sigma(-n+1)$ equally likely and  similarly from $n$ it jumps to $\sigma(n-1)$ or $\sigma(n)$ equally likely. 
\end{definition}

We suppress the dependence on $n$ of $X_{\sigma}$ for notational simplicity.
\begin{figure}[hbt]
\centering
\includegraphics[scale=.8]{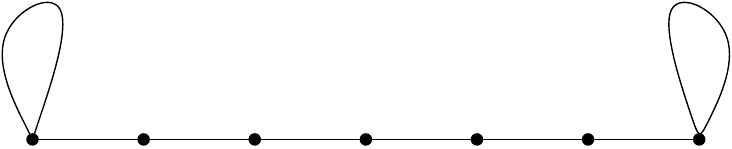}
\caption{An interval with self loops at the boundaries. 
}
\label{f.loop}
\end{figure}

The permuted walk is obtained by composing one step of the random walk on the graph in Fig.\ref{f.loop} with a permutation of the interval.
In this paper we look at the hitting times of the permuted random walk when the permutation is drawn uniformly from the set of all permutations. For any $y\in [-n,n]$ we denote by $\tau_{y}$ the hitting time of $y$. 
For any $\sigma\in S_{[-n,n]}$, denote by $\mathbb E^{\sigma}_{x}(\tau_y)$, the expectation of $\tau_{y}$, starting $X_{\sigma}$ from $x$.
We define the worst case hitting time of the walk $X_\sigma$ as  
$$E^{\sigma}=\max_{x,y\in[-n,n]}\mathbb E^{\sigma}_{x}(\tau_y).$$
Clearly $E^{\sigma}$ is a function of $\sigma.$ When $\sigma=id$ it is obvious that 
$${\E}^{id}=\mathbb E^{id}_{-n}(\tau_n)=4n^2+4n.$$ However in our main result we show that this typically is not the case.
\begin{theorem}\label{thm:mainresult}There exists a universal constant $C$ such that 
$$\mathcal{P}(E^{\sigma}<Cn)=1-o(1),$$ where $\mathcal{P}$ is the uniform measure on $S_{[-n,n]}$. 
\end{theorem}
It turns out that this problem is a special example of hitting time of the random walk on a directed expander. Random walks on undirected expanders are well understood. We state and prove their counterparts in the directed graph regime in Sections \ref{de}, \ref{ub} and \ref{lb}. For more details on undirected expanders see e.g. \cite{Markov}, \cite{hitconc} and the references therein. \\
The structure of this article is as follows: We look at the directed graph induced by the permuted walk in Section \ref{eg}. We prove expansion properties of this directed graph in Section \ref{de}. We study upper bound of the hitting time in Section \ref{ub} and prove Theorem \ref{thm:mainresult} there.
We discuss lower bounds for hitting time in Section \ref{lb}. 
\section{Eulerian Graph induced by permuted walk}\label{eg}
In this section we look at the permuted random walk as a walk on a directed graph. Recall that a directed graph is said to be Eulerian if for every vertex $v$,
$$indegree(v)=outdegree(v).$$

Given $\sigma$ we define $G_{n,\sigma}$ as a directed graph with vertex set the integers $[-n,n]$ and edge set
\begin{align*}
\{(x,\sigma(x-1)),(x,\sigma(x+1)),\mbox{ }x \notin \{-n,n\}\}
\bigcup\\\{(-n,\sigma(-n)),(-n,\sigma(-n+1)),(n,\sigma(n-1)),(n,\sigma(n))\}.
\end{align*}
Here the edge $(a,b)$ is the directed edge going from $a$ to $b$.
Since $\sigma$ is a bijection it is clear from the definition that $G_{n,\sigma}$ is a regular Eulerian digraph with $$indegree(x)=outdegree(x)=2 \,\,\,\, \forall x\in [-n,n].$$  Clearly $X_{\sigma}$ is the simple random walk on $\G$.

\begin{figure}[hbt]
\centering
\includegraphics[scale=.8]{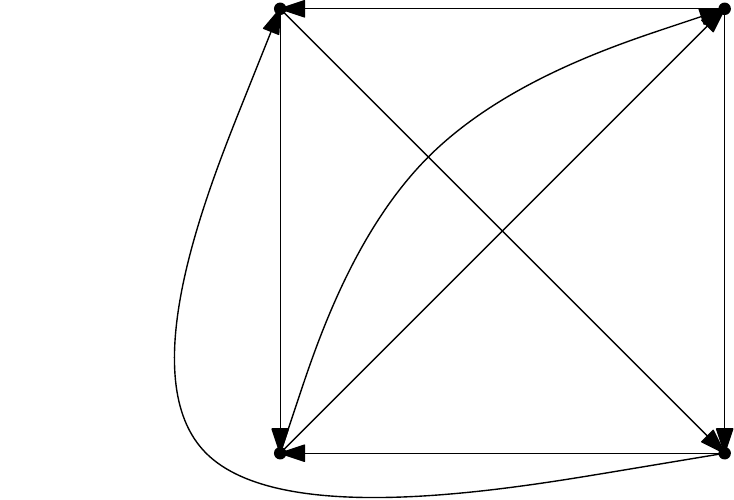}
\caption{
A $2-$ regular Eulerian. 
}
\label{eu}
\end{figure}

%
We first discuss the connectivity properties of the graphs $\G$.
\begin{definition}A directed graph $G$ is said to be connected if the graph $\tilde{G}$ obtained by removing the orientations of the edges is connected. A directed graph is said to be strongly connected if for every pair of vertices $x,y$ there is an oriented path from $x$ to $y$.
\end{definition}

In \cite{permuted} the authors prove that for all $\sigma$ the random walk on $\G$ is irreducible or in other words $\G$ is strongly connected.
We include the proof for completeness. We state a standard result about directed graphs.
\begin{lemma}\label{lemma:strongconnected}Let $G$ be a directed graph such that indegree of every vertex equals outdegree.  Then $G$ is connected implies that it is strongly connected.
\end{lemma} 
For the proof see the discussion following Theorem $12$ in \cite{graph}.
Using the lemma one can show that $\G$ is strongly connected in the following way.
\begin{lemma}[\cite{permuted}]\label{lemma:irreducible}The random walk chain on $\G$ denoted by $X_{\sigma}$ is irreducible for all $\sigma$.
\end{lemma} 
\begin{proof} By definition to prove the lemma we show that $\G$ is strongly connected. To show this, using Lemma \ref{lemma:strongconnected} it suffices to prove that $\G$ is connected. It is clear from the definition of $\G$ that for all $x \in (-n,n)$, both  $x-1$ and $x+1$ are connected to $\sigma(x)$ and hence to each other. Thus all the numbers of the same parity are connected to each other . Also the boundary point $n$ is connected to $n-1$ through $\sigma(n)$  because of the self loop at $n$. Since $n$ and $n-1$ are of opposite parity we see that $\G$ is connected.
 \end{proof}

For technical reasons we  also introduce the following graph.
\begin{definition}
Given a directed graph $G$ on the vertex set $V$, $G^2$ is a directed multigraph on $V$. The edge set of $G^2$ is constructed as follows: for $x,y \in V$ we put a distinct directed edge $(x,y)\in E(G^2)$ for every distinct directed path of length $2$ from $x$ to $y$ in $G$.
\end{definition}

We note that since $\G$ is a directed $2-$regular graph, ${\G^2}$ is a directed $4-$regular graph. The arguments in Lemma \ref{lemma:strongconnected} also show that $\G^2$ is strongly connected. In the proof of one of the subsequent results it will turn out to be more convenient to work with $\G^2$.

We quote a quadratic bound for the hitting time $\E^{\sigma}_{0}(\tau_{\{-n,n\}})$ from \cite{permuted}.
For $\sigma=id$ it is a standard result that $$\E^{id}_{0}(\tau_{\{-n,n\}})=n^2.$$ where $\tau_{\{-n,n\}}$ is the hitting time of the set $\{-n,n\}.$
In \cite{permuted} the authors show that for any $\sigma$, $$\E^{\sigma}_{0}(\tau_{\{-n,n\}})\leq 4n^2+6n+2.$$
Whether the identity permutation is the slowest, i.e. whether for all $\sigma$, the inequality $$\E^{\sigma}_{0}(\tau_{\{-n,n\}})\leq n^2,$$ holds appears as an open problem in \cite{permuted}.


We focus on typical permutations. It turns out that they are considerably faster and most permuted random walks have a linear hitting time.

\section{Bottleneck ratio and Directed Expander Graphs}\label{de}
In this section we discuss some geometric properties of certain Markov chains and the effect they have on hitting times. See Chapter $7$ from \cite{Markov} for details.


Let $P$ be the transition kernel for an irreducible Markov chain on state space $\Omega$ with invariant measure $\pi$.
The edge measure $Q$ is defined by $$Q(x,y)=\pi(x)P(x,y)$$ $$\displaystyle{Q(A,B)=\sum_{x\in A,y \in B} Q(x,y)}.$$
Thus $Q(A,B)$ is the probability of moving from $A$ to $B$ in one step when starting from stationarity.

The bottleneck ratio of a set $S$ is defined to be $$\Phi(S)=\frac{Q(S,S^{c})}{\pi(S)},$$ where $\displaystyle{\pi(S)=\sum_{x\in S} \pi(x)}$.
The bottleneck ratio of the whole chain is defined as $$\Phi_{*}=\min_{S:\pi(S)\leq1/2}\Phi(S).$$
\begin{remark}For a directed Eulerian graph 
$G=(V,E)$ and the simple random walk on $G$, the bottleneck ratio of a set $S\subset V$ is $$\Phi(S)=\frac{E(S,S^c)}{\sum_{v\in S}d_v}\hspace{60pt} (*)$$ where  $d_v$ is the degree of vertex $v$ and $E(S,S^c)$ is the number of edges going from $S$ to $S^{c}.$
\end{remark}
A family  $G_n$ of  digraphs is said to be a $(d,\alpha)$ directed expander family, if the following conditions hold.
\begin{itemize}
\item $\displaystyle{\lim_{n\rightarrow\infty}} |V(G_n)|=\infty $.
\item Both indegree and outdegree of every vertex in $G_{n}$ is $d$ for all $n.$
\item The bottleneck ratio of the simple random walk on $G_n$ satisfies $\Phi_{*}(G_n)\ge\alpha$ for all $n$.
\end{itemize}
\begin{center}\textbf{From now on all our graphs will be digraphs}.
\end{center}
If $G_n$ is $d-$regular, then the stationary measure of the random walk is uniform. Hence to show that a family of regular graphs ${G_n}$ with vertex set $V_n$ is a directed expander family, it just suffices to show that there exists a constant $c>0$ independent of $n$ such that given any set $A_n\subset V_n$ of size at most $|V_n|/2$, the number of edges going from $A_n$ to ${A^{c}_n}$ is at least $c|A_n|$.
First we introduce some definitions. Let $G=(V,E)$ be a digraph.
Let 
$$N(A,G)=\{y \in V:\, \exists \,x\in A \mbox{ such that }(x,y) \in E\},$$
and $$\partial A={\partial_{G}} A= N(A,G)\setminus A.$$
The set $N(A,G)$ is called the set of neighbors and the set $\partial A$ is called the boundary of the set $A$ in $G$.
%
%
\begin{lemma}\label{exbou}
For the chain $X_{\sigma}$ and any $A\subset[-n,n]$, $$\frac{|{\partial_{\G}}A|}{2|A|}\le \Phi(A) \le \frac{|{\partial_{\G}}A|}{|A|}. $$  
\end{lemma}
\begin{proof} Follows from $(*)$ and the following two facts:
\begin{enumerate}
\item $\G$ is $2-$regular. 
\item The number of edges going from $A$ to $A^{c}$ is at least $|\partial A|$ and at most $2|\partial A|.$
\end{enumerate}
 \end{proof} 

\begin{theorem}\label{thm:linearhitting}If $G_n$ is a $(d,\alpha)$ expander family with $|V(G_n)|=n$, then there exists a uniform constant  $C>0$ depending only on the bottleneck ratio $\alpha$, such that $\forall n$, $\forall x,y \in V(G_n)$, $$\E_x(\tau_{y})<Cn,$$ where $\E_x(\tau_{y})$ is the hitting time of $y$ for the random walk started from $x$.
\end{theorem}

Before proving the above theorem we state a standard fact about the mixing time of the random walk on expanders.
\begin{definition}Given an ergodic Markov chain with state space $\Omega$, transition kernel $P$ and stationary measure $\pi$, the mixing time of the chain is defined as
$$t_{\rm mix}:=\inf\{t: d_{TV}(P^{t}(x,\cdot),\pi)<1/4\,\, \forall\,x\in \Omega \},$$ where $d_{TV}$ is the total variation norm and $P^{t}(x,\cdot)$ is the probability distribution on $\Omega$ after running the Markov chain starting from $x$ for $t$ steps.
\end{definition}
 Recall that the lazy random walk on a graph is the Markov chain which at every time step stays at the same place with probability $1/2$ and takes a step to a uniform random neighbor with probability $1/2$. \\
We denote the lazy random walk on any regular directed graph $G$ by ${\tilde{X}}_{G}$ and on $\G$ by ${\tilde{X}}_{\sigma}.$
We state a standard result about mixing of the lazy random walk in $\mathbb{L}_{\infty}$ norm on expanders. See e.g. Section $17.4$ $(17.30)$ in \cite{Markov}.
\begin{theorem}\label{thm:supmixingexpander} If $G_n$ is a $(d,\alpha)$ expander family with $|V(G_n)|=n$, then for all $t\ge0$ $$\max_{x,y}{|P^{t}(x,y)-1/n|}\le \bigl(1-\frac{\alpha^2}{2}\bigr)^{t}.$$ Here $P^{t}(x,y)$ is the kernel of the lazy random walk after $t$ steps. 
\end{theorem}
A simple corollary of the above result is that the lazy random walk on an expander family satisfies $t_{\rm mix}=O(\log(n)).$  

\section{Upper bound on hitting times}\label{ub}
Equipped with the results in the previous section we prove Theorem \ref{thm:linearhitting}. We first recall that the Green's function for ${\tilde{X}}_{G}$, the lazy random walk on a graph $G$, is the expected amount of time the walk spends at each vertex. We define it formally below but suppress the dependence on $G$. Let 
\begin{equation} \label{z} Z_{n}(y)=\sum_{t=0}^{n} \mathbf{1}_{({\tilde{X}}_{G}(t)=y)}
\end{equation}
 be the number of visits to $y$ in the first $n$ steps. 
Here ${\tilde{X}}_{G}(t)$ denotes the position of ${\tilde{X}}_{G}$ at time $t$.

Define, $$\Gamma_{n}(x,y)=\E_{x}(Z_{n}(y)),$$ 
to be the expected number of visits to $y$ starting from $x$ in the first $n$ steps.

Hence,
     $$\Gamma_{n}(x,y)=\sum_{t=0}^{n} P^{t}(x,y),$$
where $P^{t}(\cdot,\cdot)$ is the kernel of $\tilde{X}_{G}$. \\

{\it Proof of Theorem} \ref{thm:linearhitting}. 
We first prove that for all $x$, $y$  $\in V$,  $$\Gamma_{n}(x,y)=\Theta(1),$$ where the constants involved do not depend on $n$. 
Now by Theorem \ref{thm:supmixingexpander} we know that $$1/n-\bigl(1-\frac{\alpha^2}{2}\bigr)^{t}\le P^{t}(x,y)\le 1/n+\bigl(1-\frac{\alpha^2}{2}\bigr)^{t},$$ 
for all $t$. 
The fact that $\Gamma_{n}(x,y)$ is upper bounded by a constant follows trivially from the above upper bound, since,
$$\Gamma_{n}(x,y)=\sum_{t=1}^{n}P^{t}(x,y)\le\sum_{t=1}^{n}\bigl(1/n+\bigl(1-\frac{\alpha^2}{2}\bigr)^{t}\bigr)=O(1).$$
We now show a constant lower bound. We note from the above lower bound that for some $c$ large enough just depending on $\alpha$, and $t\ge c\log(n)$
$$1/2n\le 1/n-\bigl(1-\frac{\alpha^2}{2}\bigr)^{t}\le P^{t}(x,y)$$
Hence,
 \begin{align*}
 \Gamma_{n}(x,y)=\sum_{t=1}^{c\log n-1}P^{t}(x,y)+\sum_{t=c\log n}^{n}P^{t}(x,y)\\\ge \sum_{t=c\log n}^{n}P^{t}(x,y)\ge \frac{n-c\log(n)}{2n}.
 \end{align*} 
Therefore for large enough $n$,  
$$\Gamma_{n}(x,y)\ge \frac{1}{3}.$$ 
Hence,
\begin{equation}\label{upbg}
\E_{x}(Z_{n}(y))=\Gamma_{n}(x,y)=\Theta(1).
\end{equation}
%
%
Now we will show that there exists some constant $\beta>0$ depending only on $\alpha$ such that for all $x,y \in [-n,n]$ we have $$\PP_{x}(\tilde{\tau}_{y}\le n)>\beta,$$ where $\tilde{\tau}_{y}$ is the hitting time of $y$ for the lazy chain ${\tilde{X}}_{G}$.
To show this we will bound the second moment of $Z_{n}(y)$ and use the second moment method. 

Expanding $Z_{n}^{2}(y)$ using the expression in (\ref{z}) we get,
$$Z_{n}^{2}(y)=2 \sum_{1\le i <j \le n} \mathbf{1}_{({{\tilde{X}}_{G}(i)=y, {{\tilde{X}}_{G}(j)=y)}})}+\sum_{1\le i \le n} \mathbf{1}_{({\tilde{X}}_{G}(i)=y)}.$$ 
Hence taking expectation, 
\begin{eqnarray*}\label{smb}
\E_{x}(Z_{n}^{2}(y)) &=& 2\sum_{1\le i < j \le n}P^{i}(x,y)P^{j-i}(y,y) + \Gamma_{n}(x,y) \\
									&\le&2\sum_{i=1}^{n}P^{i}(x,y)\sum_{j=1}^{n}P^{j}(y,y)+\Gamma_{n}(x,y)\\
									&=& 2\Gamma_{n}(x,y)\Gamma_{n}(y,y)+\Gamma_{n}(x,y).
\end{eqnarray*}
 From (\ref{upbg}) both $\Gamma_{n}(x,y)$ and $\Gamma_{n}(y,y)$ are $O(1).$
Therefore,
\begin{equation}\label{smm}
\E_{x}(Z_{n}^{2}(y))=O(1).
\end{equation}
Also trivially from (\ref{upbg}), $$\E_{x}(Z_{n}^{2}(y))\ge \E_{x}(Z_{n}(y))^{2} =\Theta(1).$$
 We know by the second moment method that, 
 $$\PP_{x}(Z_n(y)>0)\ge \frac{\E_{x}(Z_{n}(y))^{2}}{\E_{x}(Z_{n}^{2}(y))}.$$
 Now from (\ref{upbg}) and (\ref{smm}) we know that both the numerator and the denominator are  $\Theta(1)$ and hence we get that 
 $$\PP_{x}(\tilde{\tau}_{y}\le n)=\PP_{x}(Z_n(y)>0) \ge \beta,$$ where $\beta >0$ does not depend on the points $x$, $y$ or $n$.
This means that starting from any point $z$, the lazy walk ${\tilde{X}}_{G}$ hits $y$ in the next $n$ steps with chance at least $\beta$. Thus ${\tilde{\tau}}_{y}/n$ is dominated by a geometric random variable with success probability at least $\beta$.
Hence $$\E_{x}(\tilde{\tau_{y}})\le\frac{n}{\beta}.$$
Now for all $\,x,\,y, $ the hitting times of the simple random walk and the lazy random walk satisfies $$\E_{x}({\tau}_{y})=\frac{\E_{x}(\tilde{\tau}_{y})}{2}.$$
 Hence we are done by taking $C$ in the statement of the theorem to be $\frac{1}{\beta}$.
$\hfill\mbox{$\Box$}\smallskip$
 
The next technical result roughly states that given an uniformly chosen permutation $\sigma\in S_{[-n,n]}$, with high probability  the graphs $G_{n,\sigma}$ form an expander family. This with Theorem \ref{thm:linearhitting} will then directly imply Theorem \ref{thm:mainresult}.

\begin{theorem}\label{thm:unionbound} There exists a universal positive constant $\delta > 0 ,$ independent of $n$, such that, 
$$\mathcal{P}(\Phi_{*}(\G)>\delta)=1-o(1),$$ as $n\rightarrow \infty.$ 
\end{theorem}
\vspace{.1in}

First we state the following lemma which compares the sizes of the boundary of a subset in the graphs $\G$ and $\G^2.$
\begin{lemma}\label{comp12}
For any $A\subset [-n,n]$, 
$$|\partial_{\G^2}A|\le 3|\partial_{\G}A|.$$
\end{lemma}
\begin{proof}
Clearly $$\partial_{\G^2}A \subseteq \partial_{\G}A \cup N(\partial_{\G}A,\G).$$
Since $\G$ is $2-regular$ 
$$|N(\partial_{\G}A,\G)|\le 2|\partial_{\G}A|. $$ 
Thus  $$|\partial_{\G^2}A|\le 3|\partial_{\G}A|. $$
 \end{proof}
\begin{remark}\label{proofidea1}
Before providing formally the proof of Theorem \ref{thm:unionbound} we discuss the basic idea.\\
To show that $\G$ is an expander we use the following observation:
If $A \subset [-n,n]$ has a large number of connected components in $\mathbb{Z}$ then it has large expansion in $\G$. Otherwise, if $A$ has a small number of components in $\mathbb{Z}$ we show that for a uniformly chosen $\sigma \in S_{[-n,n]}$ the probability that the set of neighbors of $A$ in $\G$ also has a small number of connected components in $\mathbb{Z}$ is small. This then implies that the two step expansion of $A$ is likely to be high.
The number of sets $A\subset [-n,n]$ with a small number of connected components in $\mathbb{Z}$ is rather small and hence we apply union bound to argue that with high probability the two step expansion of any subset $A\subset [-n,n]$ is large in $\G$. In other words, $\G^2$ is an expander. We then use Lemma \ref{comp12} to conclude that $\G$ is an expander. 

\end{remark}
We now proceed to the proof of Theorem \ref{thm:unionbound}.
For any set $A \subset [-n,n]$, define  $$A \pm 1= \{x+1,x-1: x\in A\}\cap[-n,n].$$
\textbf{Proof of Theorem \ref{thm:unionbound}.}
By  Lemmas \ref{exbou} and \ref{comp12} it suffices to show that there exists a universal $c>0$ such that with $\mathcal{P}-$ probability  $1-o(1)$ for all $A \subset [-n,n]$ with $|A|\le n$ $$|\partial_{\G^2} A|\ge c|A|.$$ Fix $A\subset [-n,n]$ with $|A|\le n$.
By definition $$\sigma(A\pm1)\subseteq N(A,\G).$$
Since $\sigma$ is a bijection $$|N(A,\G)|\ge |A\pm1|$$ with equality occurring when $A$ does not contain any of the points $-n,n$. Thus $$|\partial_{\G}A|\ge |A\pm1|-|A|. $$
We also note that by similar reasoning as above  $$|N(A,\G^2)|\ge |\sigma(A\pm 1)\pm 1|.$$  
Let us look at the quantity $$|A\pm1|-|A|.$$ We first show that if $A$ has a large number of connected components in $\mathbb{Z}$ then $|A\pm1|-|A|$ is large.  
To avoid parity issues we look at even and odd components of a set. Formally let, $$A_{even}=\{x:x \,\,\mbox{ even and }x\in A \},$$ and similarly $A_{odd}$. 
Now say $$\displaystyle{A_{even}=\bigcup_{1}^{k_{even}}[a_i,b_i]_{even}}$$ is the decomposition into connected components of even numbers i.e. $x$ and $x+2$ are considered to be in the same component. Thus $k_{even}=k_{even}(A)$ is the number of connected components in $A_{even}$ and similarly $k_{odd}.$
Then $$A_{even}\pm 1= \bigcup_{1}^{k_{even}} [a_i-1,b_i+1]_{odd}$$ where $[a_i-1,b_i+1]_{odd}=[a_i,b_i]_{even}\pm 1$ are disjoint connected intervals of odd numbers in the interval $[-n,n]$.
Now $[a_i-1,b_i+1]_{odd}$ has one more element than $[a_i,b_i]_{even}$ unless it contains one of the points $-n,n$. Note that a component cannot contain both $-n$ and $n$ since that implies $|A|\ge n+1.$
Thus $$|A_{even}\pm 1|\ge|A_{even}|+k_{even}-2.$$
Similarly
$$|A_{odd}\pm 1|\ge|A_{odd}|+k_{odd}-2.$$
Now $A\pm 1$ is the disjoint union of $A_{odd}\pm 1$ and $A_{even}\pm 1$.
Therefore
\begin{equation}\label{key}
|A\pm 1|\ge |A|+ k_{odd}+ k_{even}-4.
\end{equation}
 
Thus given a number $\e>0,$ $|A\pm1|-|A|< \epsilon |A|$ only if 
$$k_{odd}+ k_{even}< \epsilon |A|+4.$$
Therefore if  $|A|>\frac{4}{\epsilon}$, then 
\begin{equation}\label{bad}
|A\pm1|-|A|< \epsilon |A| \implies k_{odd}+ k_{even}< 2\epsilon |A|.
\end{equation}
Hence by (\ref{bad}) the number of sets $A$ of size $m$ such that $|A\pm1|-|A|< \epsilon |A|$ for $m>\frac{4}{\epsilon}$ is clearly at most
\begin{equation}\label{number}
{2n+1\choose 2\epsilon m} {{m(1+2\epsilon)}\choose{m}}.
\end{equation}
The first factor comes from choosing the at most $2\e m$ starting points for the components in $A_{even}$ and $A_{odd}$. The second factor comes from dividing $m$ points into components; the number of possibilities is at most the number of ways to put $m$ balls in $2\e m$ bins.\\
 Let $|A|=m,$ where $\frac{4}{\e}< m \le n.$ Since $\G$ is $2-$regular it is easy to see that $$|N(A,\G)|\ge m.$$ Let $B$ be the image of the smallest $m$ elements of $A\pm1$ under $\sigma.$ 
Then the chance that $$k_{even}(B)+ k_{odd}(B)< 2\epsilon m$$ is at most  
\begin{equation}\label{map}
\frac{{2n+1\choose 2\epsilon m} {{m(1+2\epsilon)}\choose{m}}}{{2n+1 \choose m}},
\end{equation}
 from (\ref{number}) and the fact that $\sigma$ is a uniformly random permutation. Now either of the following two conditions
\begin{eqnarray*}
 k_{even}(A)+k_{odd}(A)&\ge &2\e|A|\\
k_{even}(B)+k_{odd}(B)& \ge &2\e|A|   
   \end{eqnarray*}   
 
 imply that $$\partial_{\G^2}A \ge \e|A|.$$
 This is because under the first condition $$(1+\e)|A|\le |N(A,\G)|\le|N(A,\G^2)|$$ where the first inequality follows from (\ref{key}) and the second inequality is easy to check from the definition of $\G$.
 The second condition by (\ref{key}) implies that $$(1+\e)|A|\le |N(B,\G)|\le |N(A,\G^2)|$$ where the second inequality follows from the fact that $B\subseteq N(A,\G).$
%
%

Hence by the above discussion, (\ref{number}), (\ref{map}) and union bound, 

\begin{align*}\label{sum}\pr( \exists \,A\subset[-n,n], \frac{4}{\e}<|A|\le n,|\partial_{\G^2}A|<\epsilon|A| )\\ \le  \sum_{m=1}^{n}\frac{{2n+1\choose 2\epsilon m}^2 {{m(1+2\epsilon)}\choose{m}}^{2}}{{2n+1 \choose m}}.
\end{align*}
By the next lemma for a small enough $\e$ independent of $n$ the above sum goes to $0$. Thus with $\mathcal{P}-$probability going to $1$, all sets $A$ of size between $\frac{4}{\e}$ and $n$ have at least $\e|A|$ edges going out of $A$ in $\G^2$. We also know that any set of size at most $\frac{4}{\e}$ has at least one edge going out of it since the graph is connected. Choosing a suitable $\delta$ in terms of $\e$ and using Lemma \ref{comp12} the theorem follows. \qed


\begin{lemma} $\exists\,\, \e>0$ such that 
\begin{equation}\label{sumzero}\lim_{n\rightarrow \infty}\sum_{m=1}^{n/2}\frac{{n\choose \epsilon m}^4}{{n\choose m}}= 0.
\end{equation}
\end{lemma}
\begin{proof}
To approximate the binomial coefficients we use the following upper and lower bounds,
\begin{equation*}\label{ulb}
{\bigl(\frac{n}{k}\bigr)}^{k}\le  {n\choose k} \le{\bigl(\frac{n e}{k}\bigr)}^{k}.
\end{equation*}
These bounds imply that
\begin{equation}\label{plug1}
\frac{{n\choose \epsilon m}^4}{{n\choose m}} \le \frac{{\bigl(\frac{n e}{\e m}\bigr)}^{4\e m}}{{\bigl(\frac{n}{m}\bigr)}^{m}}={\bigl(\frac{n}{m}\bigr)}^{m(4\e -1)}{\bigl(\frac{e}{\e}\bigr)}^{4\e m}. 
\end{equation}
 Choose $\e$ small enough so that $$4\e-1<-1/2 \mbox{ and }{\bigl(\frac{e}{\e}\bigr)}^{4\e}<2^{1/4}.$$
Hence from (\ref{plug1}) and the fact that $m\le n/2$, we see that
\begin{equation*}\label{plug3}
 \frac{{n\choose \epsilon m}^4}{{n\choose m}}\le {\bigl(\frac{1}{2^{1/2}}\bigr)}^{m} {\bigl(2^{1/4}\bigr)}^{m}\le{\bigl(\frac{1}{2^{1/4}}\bigr)}^{m}.
\end{equation*}
Thus
\begin{equation}\label{lem1}
\lim_{n\rightarrow \infty}\sum_{m=\sqrt{n}}^{n/2}\frac{{n\choose \epsilon m}^4}{{n\choose m}}=0.
\end{equation}

Now when $m\le \sqrt{n}$ we also have from (\ref{plug1}), 
\begin{equation}\label{plug4}
 \frac{{n\choose \epsilon m}^4}{{n\choose m}}\le {\bigl(\frac{1}{n^{1/4}}\bigr)}^{m} {\bigl(2^{1/4}\bigr)}^{m}\le{\bigl(\frac{2}{n}\bigr)}^{m/4}.
\end{equation}
Therefore
\begin{eqnarray*}
\label{sum1}
\sum_{m=1}^{n/2}\frac{{n\choose \epsilon m}^4}{{n\choose m}} \le \sum_{m=1}^{\sqrt{n}}{\bigl(\frac{2}{n}\bigr)}^{\frac{m}{4}}+\sum_{m=\sqrt{n}}^{n/2}\frac{{n\choose \epsilon m}^4}{{n\choose m}}.
\end{eqnarray*}
Together with (\ref{lem1}), this proves the lemma. 
 \end{proof}

\begin{subsection} { Proof of Theorem \ref{thm:mainresult}.}
  Using Theorem \ref{thm:unionbound} we see that with $\mathcal{P}-$probability $1-o(1)$,  ${\G}$ has bottleneck ratio lower bounded by some fixed positive constant $\delta$.  Hence using Theorem \ref{thm:linearhitting} we get that there exists a universal constant $A$ depending on $\delta$ such that with $\mathcal{P}-$probability $1-o(1)$ the expected hitting time of $y$ starting from $x$ for the random walk $X_{\sigma}$ on $\G$ is less than $An$ for all $x,y \,\in[-n,n]$. 
Hence we are done. 
\qed
\end{subsection}

\section{Typical lower bound for hitting times}\label{lb}
In this section we show that for a given pair $x,y \in [-n,n]$, with high probability, the expected hitting time from $x$ to $y$  is $\Omega(n)$. We however cannot hope to make a stronger statement like: There exists a universal constant $D>0$ such that with high probability, for all $x,y \in [-n,n]$  $$\E^{\sigma}_{x}(\tau_y)> Dn.$$  
This is because of the well known result, see \cite{Ta}, which says that if $\sigma$ is a uniformly chosen random permutation then
$$\#\{x:\sigma(x+1)=x\}\stackrel{law}{\rightarrow} Y,$$
 where $Y$ is a Poisson random variable with mean $1$.
This implies that with probability lower bounded away from $0$ the graph $\G$ has self loops. Let $x$ be a point depending on $\sigma$ such that $\sigma(x+1)=x$. Then we see that $$\E^{\sigma}_{x}(\tau_{\sigma(x-1)})=2.$$
\begin{theorem}\label{thm:typical} 
$$\min_{x,y \in [-n,n]}\mathcal{P}(\E^{\sigma}_{x}(\tau_{y})\ge \frac{n}{18})=1-o(1).$$
\end{theorem}

First we look at the lazy chain ${\tilde{X}}_{\sigma}$ and the case when we start from stationarity.
\begin{lemma}\label{lemma:stationarity} Starting with the stationary distribution $\pi$, for any $y\in[-n,n],$
\begin{equation}\label{eqn:lower}\PP_{\pi}({\tilde{\tau}}_{y}\ge n/3)\ge1/2,
\end{equation}
where $\tilde{\tau}_{y}$ is the hitting time of $y$ for the lazy chain ${\tilde{X}}_{\sigma}.$
\end{lemma}
\begin{proof}
$$\PP_{\pi}({\tilde{\tau}}_{y}\le n/3)\le \sum_{t=1}^{n/3}\PP_{\pi}({\tilde{X}}_{\sigma}(t)=y)=\frac{n}{3(2n+1)}.$$
%
 \end{proof}
We abbreviate $P^{t}(x,\cdot)$, the probability measure on $\G$ after running ${\tilde{X}}_{\sigma}$ starting from $x$ for $t$ steps by $\mu_{x}^{t}.$ Then we have the following corollary.

\begin{corollary}\label{linearfromlog} There exists a universal constant $\gamma>0$ such that with $\mathcal{P}-$probability $1-o(1)$, $\G$ satisfies:  for all  $x,y\in[-n,n]$ and $t\ge\gamma\log(n)$,
\begin{align*}
\PP_{\mu_{x}^{t}}({\tilde{\tau}}_{y}\ge n/3)\ge1/3.
\end{align*}
\end{corollary}
The corollary says that with high probability $\sigma$ would be such that starting with distribution $\mu_{x}^{t}$, the hitting time of $y$ for the chain ${\tilde{X}}_{\sigma}$ is at least $n/3$ with chance at least $1/3$, uniformly for all $x,y$ and $t\ge\gamma \log(n)$.
\begin{proof}By Theorems \ref{thm:supmixingexpander} and \ref{thm:unionbound} we know that there is a universal constant $\gamma>0$ such that with $\mathcal{P}-$probability $1-o(1)$, for all $x \in [-n,n]$
            $$d_{TV}(\mu_{x}^{k},\pi)\le1/6$$ for all $k\ge\gamma \log(n)$. 
Fix $t\ge\gamma \log(n)$.
 We have for all $x,y\in [-n,n]$
\begin{eqnarray*}\label{eqnarray:loweranystarting}\PP_{\mu_{x}^{t}}({\tilde{\tau}}_{y}\ge n/3)\ge \PP_{\pi}({\tilde{\tau}}_{y}\ge n/3)-d_{TV}(\mu_{x}^{t},\pi).
\end{eqnarray*}
The above follows from the definition of total variation norm. Since
$$\PP_{\pi}({\tilde{\tau}}_{y}\ge n/3)-\PP_{\mu_{x}^{t}}({\tilde{\tau}}_{y}\ge n/3)$$
\begin{align*}
&=\sum_{z\in[-n,n]}(\pi(z)-\mu_{x}^{t}(z))\PP_{z}({\tilde{\tau}}_{y}\ge n/3)\\
&\le \sum_{\stackrel {z\in[-n,n]}{\pi(z)-\mu_{x}^{t}(z)>0}}(\pi(z)-\mu_{x}^{t}(z))=d_{TV}(\mu_{x}^{t},\pi).
 \end{align*}
Hence from (\ref{eqn:lower}), we conclude that
$$\PP_{\mu_{x}^{t}}({\tilde{\tau}}_{y}\ge n/3)\ge1/2-1/6.$$ 

 \end{proof}


Now we show that starting from $x$ with probability lower bounded away from $0$ the chain does not hit $y$ in the first $\Omega(\log (n))$ steps.\\

Let $d_{\sigma}(x,y)$ be the minimum length of all directed paths from $x$ to $y$ in $\G.$ We call $d_{\sigma}(x,y)$ as the distance of $y$ from $x$ in $\G.$
\begin{lemma}\label{notinlog}There exists a universal constant $k>0$ such that given $\alpha>0$ with $\mathcal{P}-$probability $1-o(1)$ for all $x,y \in [-n,n]$ with $d_{\sigma}(x,y)>k$, 
$$\PP_{x}({\tilde{\tau}}_{y}\le \alpha\log(n))\le \frac{1}{4}.$$ 
\end{lemma}

\begin{proof}By Theorems \ref{thm:supmixingexpander} and \ref{thm:unionbound} we know that there exists an $\e>0$ such that with $\mathcal{P}-$probability $1-o(1)$, 
\begin{equation}\label{equation:geometricdecay} P^{t}(x,y)\le \frac{1}{2n+1}+\bigl(1-\frac{\e^2}{2}\bigr)^{t},
\end{equation}
 for all $x,y\in [-n,n]$, where $P^{t}$ is the kernel of $\tilde{X}_{\sigma}$.

Now clearly  $$\PP_{x}({\tilde{\tau}}_{y}\le \alpha \log(n))\le \sum_{t=1}^{\alpha \log(n)}P^{t}(x,y).$$
Since by hypothesis $d_{\sigma}(x,y)>k$, we know $P^{t}(x,y)=0$ for all $t\le k$.
Hence by (\ref{equation:geometricdecay}), $$\PP_{x}(\tilde{\tau}_{y}\le\alpha\log(n))\le\frac{\alpha\log(n)}{n}+\sum_{t=k}^{\infty}\bigl(1-\frac{\e^2}{2}\bigr)^{t}<\frac{1}{4},$$ for large $n$ if $k$ is chosen a priori to be large enough just depending on $\e$. 
 \end{proof}
The next lemma proves that given $x,y \in[-n,n]$, the distance to $y$ from $x$ is typically large in $\G$. 

\begin{lemma}\label{far}For every positive integer $k$ 
$$\min_{x,y \in [-n,n]}\mathcal{P}(d_{\sigma}(x,y)>k)=1-o(1)$$
as $n\rightarrow \infty.$
\end{lemma}
\begin{proof}
Given $w\in [-n,n]$, for any positive integer $r$ we denote by $\mathbb{B}_{r}(w)$ the set of all points $z\in [-n,n]$ such that there is a directed path of length at most $r$ from $w$ to $z$ in $\G$. 
Fix a positive integer $k$. Also fix $x,y \in [-n,n].$ We show that $y$ does not lie in $\mathbb{B}_{k+1}(x)$ with $\mathcal{P}-$probability $1-o(1)$. Now $\sigma$ is a uniformly chosen permutation. We choose $\sigma$ by exposing the points in $B_{l}(x)$ for $l=1,2,\ldots$. Then we look at the chance that $y\notin \mathbb{B}_{l}(x)$ for $l=1,2,\ldots k+1$. We do this by recursion. Suppose we have exposed all points in $\mathbb{B}_{l}(x)$ and $y \notin \mathbb{B}_{l}(x)$. Now $\partial\mathbb{B}_{l}$ is the image of the set $\partial\mathbb{B}_{l-1}\pm1$ under $\sigma$. We look at the probability that the image set does not contain $y$. Now clearly since the graph is $2-$ regular  $$|\partial\mathbb{B}_{l-1}|\le 2^{l}.$$ 
Thus there are at most $2^{l+1}$ points to be exposed and they can take values in the set $[-n,n]\setminus \mathbb{B}_{l}(x) $ which has cardinality at least $2n-2^l$. So the chance that $y\notin \mathbb{B}_{l+1}(x)$ given $y\notin \mathbb{B}_{l}(x)$ is at least,
$$\Psi_{n,l}=\prod_{i=0}^{2^{l+1}}\frac{(2n-2^{l}-1-i)}{2n - 2^{l}-i}.$$

Clearly for a fixed $l,$ $$\lim_{n\rightarrow \infty} \Psi_{n,l}=1.$$
Hence the probability that $y\notin{\mathbb{B}_{k+1}(x)}$ is lower bounded by  
\begin{align*}\prod_{l=1}^{k+1}\Psi_{n,l}=1-o(1).
\end{align*}
Thus we are done.
 \end{proof}

\textit{Proof of Theorem \ref{thm:typical}}. 
Choose $\alpha=2\gamma$ where $\alpha$ appears in Lemma \ref{notinlog} and $\gamma$ appears in the statement of Corollary \ref{linearfromlog}. 
First we observe that
$$\PP_{x}(\tilde{\tau}_{y}\ge n/3)>\PP_{\mu_{x}^{\alpha \log n}}(\tilde{\tau}_{y} \ge n/3)-\PP_{x}(\tilde{\tau}_{y}\le \alpha \log(n)).$$
This follows from the fact:
\begin{align*}\{\tilde{\tau}_{y}>\alpha\log(n)\} \cap \{\mbox{the first hitting time of } y \\\mbox{ after time } \alpha\log(n)\,\ge \,n/3\}\\ \subset \{\tilde{\tau}_{y} \ge n/3\},
\end{align*}
and the simple relation $$P(A\cap B)\ge P(A)-P(B^c).$$
By the choice of $\alpha$ and Corollary \ref{linearfromlog} we get that with $\mathcal{P}-$ probability $1-o(1)$ 
$$\PP_{\mu_{x}^{\alpha \log n}}(\tilde{\tau}_{y}\ge n/3)\ge 1/3.$$
Now by Lemmas \ref{notinlog} and \ref{far} we get that with $\mathcal{P}-$ probability  $1-o(1)$
 $$\PP_{x}(\tilde{\tau}_{y}\le \alpha \log(n))\le 1/4.$$
Thus $$\PP_{x}(\tilde{\tau}_{y} \ge n/3)\ge 1/3-1/4.$$ 
Hence with $\mathcal{P}-$probability $1-o(1)$  $$\E^{\sigma}_{x}(\tilde{\tau}_y)\ge \frac{n}{36}.$$
which then implies that $$\E^{\sigma}_{x}({\tau}_y)\ge \frac{n}{18}.$$
\qed
\begin{remark}The proof of Theorem \ref{thm:typical} immediately generalizes to show that given a point $x \in (-n,n)$ the hitting time of $\{-n,n\}$ starting from $x$ on $\G$ is $\Omega(n)$ with $\mathcal{P}-$probability going to $1$ as $n$ goes to infinity. We omit the details.
\end{remark}

\section*{Acknowledgment}
We thank Perla Sousi for careful reading of the manuscript and making helpful comments. 
We also thank Christopher Hoffman for useful discussions. Part of this work was done while the first author was visiting the Theory Group at Microsoft Research Redmond. 

\end{document}